\documentclass[12pt,a4paper]{article}
\usepackage{latexsym,epsfig}
\usepackage{amsmath,amssymb,amsthm}
\usepackage{subfigure}

\newtheorem*{thmnoname}{Theorem}

\newtheorem{Theorem}{Theorem}[section]
\newtheorem{lemma}{Lemma}

\newtheorem{definition}{Definition}
\newtheorem{obs}{Observation}

\newtheorem{claim}{Claim}[section]
\newtheorem{problem}{Problem}
\newtheorem{ih}{IH}

\newcommand{\ol}{\overline}

\title{Coloring half-planes and bottomless rectangles}

\date{}
\author{Bal\'azs Keszegh\thanks{R\'enyi Institute, Hungary, {\tt keszegh@renyi.hu}~~
        This work was done while visiting the School of Computing Science at Simon
                                Fraser University.}}

\begin{document}
\maketitle

\begin{abstract}
We prove lower and upper bounds for the {\em chromatic number} of certain {\em hypergraphs} defined by geometric regions. This problem has close relations to {\em conflict-free colorings} \cite{evenlotker}. 
One of the most interesting type of regions to consider for this
problem is that of the {\em axis-parallel rectangles}. We completely solve
the problem for a special case of them, for {\em bottomless
rectangles}. We also give an almost complete answer for {\em half-planes} and pose
several open problems. Moreover we give efficient coloring algorithms.
\end{abstract}

\section{Introduction} \label{intro}
Given a set of points $P$ and a family of regions $\cal F$, there is a natural way to associate two hypergraphs to them, dual to each other.
Following the notations of \cite{tardosmanu}, $H(P,\cal F)$ denotes the hypergraph on the vertex set $P$, whose hyperedges are all subsets of $P$ that can be obtained by intersecting $P$ with a member of $\cal F$. The dual hypergraph $H^*(P,\cal F)$ denotes the hypergraph on the vertex set $\cal F$, for each $p\in P$ it has a hyperedge consisting of all regions in $\cal F$ that contain $p$. In the dual case we are usually interested in the case when $P$ equals to $P_{plane}$, the set of all points of the plane.

A {\em proper coloring} of a hypergraph is a coloring of its vertex set $V(H)$ such that no hyperedge with at least two vertices is monochromatic. The {\em chromatic number} of $H$ is the smallest number of colors for which there exists a proper coloring of $H$. Let $H_{k}$ (resp. $H_{\ge k}$) denote the hypergraph on the vertex set $V(H)$, consisting of all $k$-element (resp. at least $k$-element) hyperedges of $H$. Note that by definition $\chi (H)=\chi (H_{\ge 2})$. By slightly abusing the notation, given a hypergraph $H$, we call a coloring of $V(H)$ a {\bf $k$-proper coloring of $H$} if it is a proper coloring of $H_{\ge k}$. We are interested in determining $\chi(H_{\ge k}(P,\cal F))$ and $\chi(H^*_{\ge k}(P,\cal F))$ for certain families. 
The study of the chromatic number of such hypergraphs was stimulated by {\em conflict-free colorings} (\cite{smorodnew}). We discuss this relationship in Section \ref{cfsubsec}.
Regions for which proper colorings of the corresponding hypergraphs and conflict-free colorings has been studied, include discs (\cite{evenlotker}, \cite{smorodnew}, etc.) and axis-parallel rectangles (\cite{chenpach}, \cite{pachtoth}, \cite{elbassioni}, etc), for a more detailed overview of these results see Section \ref{secdiscussion}.

Throughout the paper, in the primal case we consider only finite point sets. In the dual case $P$ is usually equal to $P_{plane}$, the set of all points of the plane and we consider only a finite set of regions $\cal F$.

Now we can summarize the results presented in this paper. For most regions we study, the chromatic number of $H_{\ge k}({\cal F}, P)$ (resp. of $H^*_{\ge k}({\cal F}, P_{plane})$) can always be bounded from above by a constant independent of $P$ (resp. of $\cal F$). Our aim is to determine the best possible upper bound for certain families and any given $k$. In Section \ref{secbottomless} we solve all cases for {\em bottomless rectangles} (or {\em vertical half-strips}), a special case of axis-parallel rectangles. A {\bf bottomless rectangle} is the set of points $\{(x,y)| a<x<b,y<c\}$ for some $a,b$ and $c$. The set of all bottomless rectangles is denoted by $\cal B$. For our coloring purposes the family of bottomless
rectangles is equivalent with the family of (ordinary) axis-parallel rectangles having their lower edge on a common horizontal {\em base-line} $e$. 
The two main results are the following.

\begin{thmnoname}(Theorem \ref{dualbrect}) Any finite family of bottomless rectangles can be colored with $3$ colors such that any point contained by at least $2$ of them is not monochromatic.
\end{thmnoname}
\begin{thmnoname}(Theorem \ref{dualbrect2})
Any finite family of bottomless rectangles can be colored with $2$ colors such that any point contained by at least $3$ of them is not monochromatic.
\end{thmnoname}
At the end of the section we deduce results for another very similar special case, the set of rectangles {\em intersecting} a common horizontal base-line $e$ (denoted by $\cal B'$).

In Section \ref {sechp} we prove theorems which give an almost complete answer for {\em half-planes} (the set of all half-planes is denoted by $\cal H$). The two main results are the following.
\begin{thmnoname}(Theorem \ref{hptheorem2}) Any finite set of points can be colored with $2$ colors such that any half-plane containing at least $3$ of them is not monochromatic.
\end{thmnoname}
\begin{thmnoname}(Theorem \ref{halfplanes})
Any finite family of half-planes can be colored with $3$ colors such that any point contained by at least $2$ of them is not monochromatic.
\end{thmnoname}

\begin{table}[t]
\begin{center}
\setlength{\baselineskip}{5cm}
\begin{tabular}{|c|c|c|c|c|c|c|c|c|c|}
\hline
\raisebox{-0.1cm}[0.3cm][0.3cm]{$k$}           & ~~~\raisebox{-0.1cm}{2}~~~ & \raisebox{-0.1cm}{3} & ~~\raisebox{-0.1cm}{$\geq$4}~~ \\
\hline
\hline
\raisebox{-0.1cm}[0.3cm][0.3cm]{$\chi_k({\cal B})$}     &  \raisebox{-0.1cm}{\em 3} \raisebox{-0.1cm}{(Claim \ref{folklore})}& \raisebox{-0.1cm}{{\bf 3} (Claim \ref{wcfbthm})} & \raisebox{-0.1cm}{\bf 2} \raisebox{-0.1cm}{(Thm \ref{wcfbthm2})} \\
\hline
\raisebox{-0.1cm}[0.3cm][0.3cm]{$\ol{\chi_k}({\cal B})$}     & \raisebox{-0.1cm}{\bf 3} \raisebox{-0.1cm}{(Thm \ref{dualbrect})} & \raisebox{-0.1cm}{\bf 2} \raisebox{-0.1cm}{(Thm \ref{dualbrect2})}& \raisebox{-0.1cm}{\em 2} \\
\hline

\raisebox{-0.1cm}[0.3cm][0.3cm]{$\chi_k({\cal H})$}    &  \raisebox{-0.1cm}{\bf 4} \raisebox{-0.1cm}{(Thm \ref{hptheorem})}& \raisebox{-0.1cm}{\bf 2} \raisebox{-0.1cm}{(Thm \ref{hptheorem2})}& \raisebox{-0.1cm}{\em 2} \\
\hline
\raisebox{-0.1cm}[0.3cm][0.3cm]{$\ol{\chi_k}({\cal H})$}  & \raisebox{-0.1cm}{\bf 3} \raisebox{-0.1cm}{(Thm \ref{halfplanes})}& \raisebox{-0.1cm}{\em 2 or 3} & \raisebox{-0.1cm}{\bf 2}\raisebox{-0.1cm}{(Thm \ref{halfplanes2})}\\

\hline

\end{tabular}
\caption{table of results about $\cal B$ and $\cal H$} \vspace{-0.4cm}
\label{table:results}
\end{center}
\end{table}

\begin{table}[t]
\begin{center}
\setlength{\baselineskip}{5cm}
\begin{tabular}{|c|c|c|c|c|c|c|c|c|c|}
\hline
\raisebox{-0.1cm}[0.3cm][0.3cm]{$k$}           & ~~~\raisebox{-0.1cm}{2}~~~ & \raisebox{-0.1cm}{3$\ldots$6} & ~~\raisebox{-0.1cm}{$\geq$7}~~  \\
\hline
\hline

\raisebox{-0.1cm}[0.3cm][0.3cm]{$\chi_k({\cal B'})$}     & \raisebox{-0.1cm}{\bf 3$\leq$} \raisebox{-0.1cm}{(Claim \ref{bprimel});} \raisebox{-0.1cm}{\bf $\leq$6}\raisebox{-0.1cm}{(Claim \ref{bprime2})} & \raisebox{-0.1cm}{{\em2$\leq$}; {\em$\leq$}{\bf 3}} \raisebox{-0.1cm}{(Claim \ref{bprime3})} & \raisebox{-0.1cm}{\bf 2} \raisebox{-0.1cm}{(Claim \ref{bprime4})}\\
\hline
\raisebox{-0.1cm}[0.3cm][0.3cm]{$\ol{\chi_k}({\cal B'})$}     &  \raisebox{-0.1cm}{{\bf 4$\leq$}} \raisebox{-0.1cm}{(Claim \ref{dualbrectprime}); {\em $\leq$8} \cite{smorodnew}} & \multicolumn{2}{|c|}{{\raisebox{-0.1cm}{{\em 3$\leq$} \cite{pachtothtardos}; {\bf $\leq$4}}\raisebox{-0.1cm}{(Claim \ref{bprimed})}}
}  
\\
\hline
\end{tabular}
\caption{table of results about ${\cal B'}$} \vspace{-0.4cm}
\label{table:bottomless}
\end{center}
\end{table}
In Table \ref{table:results} and Table \ref{table:bottomless} we summarize these results (for the notations see Definition \ref{defchi}), the bold ones are proved in Sections \ref{secbottomless} and \ref{sechp}, others come from monotonicity on $k$ (see Observation \ref{order}) except for $\chi_2({\cal B})=3$ (folklore), $\ol{\chi_2}({\cal B'})\leq 8$ \cite{smorodnew} and that $\ol\chi_k({\cal B})\ge 3$ for every $k$ \cite{pachtothtardos}. 

For reasons of convenience, we introduce the following notations.
\begin{definition} \label{defchi}
Given a family $\cal F$ of planar regions and a finite set of points $P$,
\begin{itemize}
\item
$\chi_k({\cal F},n)=max_{P'\subset {P_{plane}};{|P'|=n}}\chi(H_{\ge k}(P',\cal F))$,
\item 
if $\cal F$ is infinite, $\chi_{k}({\cal F})=max_{n} \chi_k({\cal F},n)$, if it exists,
\item
$\ol{\chi_k}({\cal F},n)=max_{{\cal F'}\subset {\cal F};|{\cal F'}|= n}\chi(H_{\ge k}^*({P_{plane}},{\cal F'}))$,
\item
if $\cal F$ is infinite, $\ol{\chi_{k}}({\cal F})=max_{n}{\ol \chi_k} ({\cal F},n)$, if it exists.
\end{itemize}
\end{definition}
\begin{obs} \label{order}
$\chi_k({\cal F},n)\leq \chi_l({\cal F},n)$ and $\chi_k({\cal F})\leq \chi_l({\cal F})$, if $k\geq l$.
\end{obs}
\begin{proof}
For any given $\cal F$ and $P'$, $\chi (H_{\ge k}(P',{\cal F}))\le \chi(H_{\ge l}(P',\cal F))$, because $H_{\ge k}(P',{\cal F})\subset H_{\ge l}(P',\cal F)$ by definition. Thus, the same inequality holds if we take the maximum over all point sets $P'$ of size $n$ and also if we take the maximum over all $n$.  
\end{proof}

Slightly modifying the notation of \cite{har} we call a family of regions $\cal F$ {\bf monotone} if, for any finite $P, F\in \cal F$ and $l$ positive integer, if $F$ contains at least $l$ points of $P$, then there exists $F'\in \cal F$ containing exactly $l$ points of $P$, all contained by $F$ as well. The following is an easy consequence of the definition.

\begin{obs} \label{monotone}
If $\cal F$ is a monotone family of regions, then for any finite $P$, $\chi(H_{\ge k}(P, {\cal F}))=\chi(H_k(P,\cal F))$, thus in the definition of $k$-proper coloring it is enough to restrict our condition to regions containing exactly $k$ points of the point set.
\end{obs}
Note that monotonicity could be defined in the dual version as well, but none of the types of regions we study are monotone in that dual sense.

\subsection{Relation to conflict-free colorings}\label{cfsubsec}
Motivated by a frequency assignment problem in cellular telephone networks, Even, Lotker, Ron and Smorodinsky \cite{evenlotker} studied the following problem. 
Cellular networks facilitate communication between fixed {\em base stations} and moving {\em clients}. Fixed frequencies are assigned to base-stations to enable links to clients. Each client continuously scans frequencies in search of a base-station within its range with good reception. The fundamental problem of frequency assignment in cellular networks is to assign frequencies to base-stations such that every client is served by some base-station, i.e., it can communicate with a
station such that the frequency of that station is not assigned to any other station it could also communicate with (to avoid interference). Given a fixed set of base-stations we want to minimize the number of assigned frequencies.

First we assume that the ranges are determined by the clients, i.e., if a base-station is in the range of some client, then they can communicate.
Let $P$ be the set of base-stations and $\cal F$ the family of all possible ranges of any client.
Given some family $\cal F$ of
planar regions and a finite set of
points $P$
we define {\bf $cf({\cal F},P)$} as the smallest number of colors which are enough to color the points of $P$ such that in every region of $\cal F$ containing at least one point, there is a point whose color is unique among the points in that region. The maximum over all point sets of size $n$ is the so called {\bf conflict-free coloring number} (cf-coloring in short), denoted by {\bf $cf({\cal F},n)$} (for a summary of the definitions of the different versions of $cf$-colorings see Definition \ref{defcf}). Determining the cf-coloring number for different types of regions $\cal F$ is the main aim in this topic. 

We can define a dual version of conflict-free colorings as well. It is natural to assume that the ranges are determined by the base-stations, i.e., if a client is in the range of some base-station, then they can communicate. For a {\em finite} family of planar regions $\cal F$
we define {\bf $\ol{cf}(\cal F)$} as the smallest number of colors which
is enough for coloring the regions of $\cal F$ such that for every point in
$\cup\cal F$ there is a region whose color is unique among the colors of the regions containing it. For a (not necessarily finite) family $\cal F$ of planar regions let $\ol{cf}({\cal F},n)$, the {\bf conflict-free region-coloring number of $\cal F$} be the maximum of $\ol{cf}({\cal
F'})$ for $\cal F'\subseteq {\cal F}$, $|{\cal F'}|=n$.

Smorodinsky \cite{smorodthesis} and then Har-Peled et al. \cite{har} defined generalized versions of these notions. A {\em $cf_k$-coloring} of a point set is a coloring such that for each region $F$ of $\cal F$ containing at least one point, there is a color which is assigned to at most $k$ points contained by $F$. Note that a $cf$-coloring is actually a $cf_1$-coloring. The region-coloring version is defined similarly.

\begin{definition} \label{defcf}
Given a family $\cal F$ of planar regions and a finite set of points $P$,
\begin{itemize}
\item
$cf_k({\cal F},P)=\min c$: $\exists$ $c$-coloring $f$ of $P$ s.t. $\forall F\in \cal F$ $\exists x$ s.t. $1\leq|\{p: p\in F, f(p)=x\}|\leq k$,
\item
$cf_k({\cal F},n)=\max_{|P|=n} cf_k({\cal F},P)$,
\item
if $\cal F$ is finite, $\ol {cf_k}({\cal F})=\min c$: $\exists$ $c$-coloring $f$ of $\cal F$ s.t. $\forall p\in \cup \cal F$ $\exists x$ s.t. $1\leq|\{F: p\in 
F, f(F)=x\}|\leq k$,
\item
$\ol {cf_k}({\cal F},n)=\max_{|{\cal F'}|=n, \cal F'\subset F} \ol{cf_k}({\cal F'})$.
\end{itemize}
\end{definition}

\begin{obs} \label{cfwcf}
For $k\ge 2$, $\chi_k({\cal F},n)\leq cf_{k-1}({\cal F},n)$.
\end{obs}
\begin{proof}
For any given $\cal F$ and $P$ $\chi_k({\cal F},P)\le cf_{k-1}({\cal F},P)$,because a $cf_{k-1}$-coloring is also a $k$-proper coloring by definition. Thus, the same inequality holds if we take the maximum over all point sets $P$ of size $n$.  
\end{proof}

Even et al. \cite{evenlotker} presented a general algorithmic framework on conflict-free colorings, refined version of this approach appeared in \cite{har} (and later in \cite{smorodnew}) where they summarized it in a lemma showing essentially that the chromatic number yields a good upper bound to the conflict-free coloring number (they deal only with the case $k=2$). The algorithm giving a conflict-free coloring from proper colorings is the following.
In each step take a biggest color class in a proper coloring of the point set. After coloring it to a new color, delete it and do the same for the new (smaller) point set. 
This framework also works for $k>2$ as it is easy to see that taking in each step a color class of a $k$-proper coloring we get a $cf_{k-1}$-coloring.
The generalized version of the lemma stated in \cite{har} is as follows.

\begin{lemma} \label{har}
For any fixed $k\ge2$
\item[(i)]
if $\chi_k({\cal F},n)\leq c$ for some constant $c$, then $cf_{k-1}({\cal F},n)\leq \frac {\log n}{\log(c/(c-1))}=O(\log n)$,
\item[(ii)]
if $\chi_k({\cal F},n)=O( n^{\epsilon})$ for some $\epsilon>0$, then $cf_{k-1}({\cal F},n)=O( n^{\epsilon})$.
\end{lemma}

Observation \ref{cfwcf} and Lemma \ref{har} show that $\chi_k$ and $cf_{k-1}$ are usually
close to each other. Often the best known bound for $cf$ is obtained
using Lemma \ref{har}. This is the main motivation why we want to determine the chromatic number for different types of regions.
As this lemma gives bound for $cf_{k-1}$ using $\chi_k$, it motivates the investigation of $k$-proper colorings for $k>2$.
The dual version of Lemma \ref{har} holds as well.

We briefly mention the relation of proper colorings to cover decompositions. For more on this, see \cite{kbphd}.
Following the notation of \cite{pachtothtardos} a partial {\em $k$-fold} covering of the plane with a set of regions $\cal F$ is {\em decomposable} if we can partition the set into two subsets such that for any point covered by $\cal F$ at least $k$ times, there is a region in each part covering this point. Clearly, $\ol{\chi_k({\cal F})}=2$ is equivalent to this as a proper coloring of $\cal F$ gives a good partition and vice versa. \\

\section{Bottomless Rectangles} \label{secbottomless}
\subsection{Coloring points}

From now on we assume that there are no two points with the same $x$- or $y$-coordinate. It is easy to show that if this is not the case, then coloring the point set after a small perturbation gives a needed coloring for the original point set as well. Recall that a bottomless rectangle is the set of points $\{(x,y)| a<x<b,y<c\}$ for some $a,b$ and $c$ and that we can regard bottomless
rectangles as axis-parallel rectangles having their lower edge on a common horizontal {\em base-line} $e$. Indeed in the primal and dual case as well we can choose an $e$ low enough (below all the points when we color points, below all the top edges of any of the bottomless rectangles when we color bottomless rectangles) so that cutting all the bottomless rectangles by $e$ does not make any difference in the definitions of the colorings.
In this section {\em upwards order} means the ordering of points according to their $y$-coordinate starting with the point having the smallest $y$-coordinate (the {\em lowest} point). We refer to the $x$-coordinate left to right order as $x$-order. Given a point set $P$ and a $p\in P$, $p$'s {\em left neighbor} is the point in $P$ with the biggest $x$-coordinate which is smaller than $p$'s $x$-coordinate. Similarly we can define the {\em right neighbor} of $p$. Also, two neighboring points are called $x$-adjacent. For a point set $P$ a sequence of consecutive points of $P$ in the $x$-order is called an {\em $x$-interval}. 
 
The proof of the following, rather trivial result is just presented for the sake of completeness.
\begin{claim} \label{folklore}
{\sl (folklore)}
$\chi_2({\cal B})=3$, i.e., any finite set of points can be colored with $3$ colors such that any bottomless rectangle containing at least $2$ of them is not monochromatic.
\end{claim}
\begin{proof}
First we prove that $\chi_2({\cal B})\leq 3$. We want to color a point set $P$ with $3$ colors such that any bottomless rectangle containing at least $2$ points contains two differently colored points.
The proof is by induction on the number of points in $P$.
First we color the lowest point of $P$ arbitrarily with one of the three colors then we color the points one by one in upwards order. We always maintain the following induction hypothesis.

\begin{ih}\label{ih1}
Among the points $P'$ already colored there are no two $x$-adjacent points with the same color.
\end{ih}

This can be done as $p$ lies between (at most) two already colored points in the $x$-order by the assumption that no two $x$-coordinates are equal. Thus for $p$ one can choose a color different from the (at most) two neighbors of $p$ (in $P'$), so that no two $x$-adjacent points of $P'$ have the same color.

In this way any bottomless rectangle $B$ containing at least two points contains two differently colored points as well, i.e., this is a $2$-proper coloring. Indeed, suppose on the contrary that all the points in $B$ are the same color. Then in the step when we considered the highest point $p$ of $B$, there were two $x$-adjacent points in $P'$ with the same color, contradicting IH \ref{ih1}.

The lower bound $\chi_2({\cal B})\geq 3$ follows from the fact that for example the points with coordinates (0,0), (1,1) and (2,0) cannot be colored with 2 colors in a proper way, since any two of them can be cut off by a bottomless rectangle.
\end{proof}
\begin{claim} \label{wcfbthm}
$\chi_{3}({\cal B})=3$.
\end{claim} 
\begin{proof}
Using Observation \ref{order} with Claim \ref{folklore} we know that $\chi_3({\cal B})\leq \chi_{2}({\cal B})=3$.
Thus, it is enough to prove that $\chi_{3}({\cal B})>2$. For that we show that the $12$ point construction in Figure \ref{bconstr} cannot be colored with $2$ colors such that any bottomless rectangle containing at least $3$ points contains two differently colored points. Suppose on the contrary that there is such a coloring. Denote the points ordered by their $x$-coordinate from left to right by $p_1,p_2,\ldots,p_{12}$. Among the points $p_4,p_5,p_6$ there are two with the same color, wlog. assume that this color is red. If $p_4$ and $p_5$ are red, then all of $p_1,p_2,p_3$ are blue as there is a bottomless rectangle containing only $p_4, p_5$ and any one of these $3$ points. This is a contradiction as there is a bottomless rectangle containing only these $3$ points, all blue. If $p_4$ and $p_6$ are red then similar argument for the points $p_{10},p_{11},p_{12}$, if $p_5$ and $p_6$ are red then similar argument for the points $p_7,p_8,p_9$ leads to a contradiction.
\begin{figure}
    \centering
    \subfigure [Claim \ref{wcfbthm}]{ \label{bconstr}
            {\epsfig{file=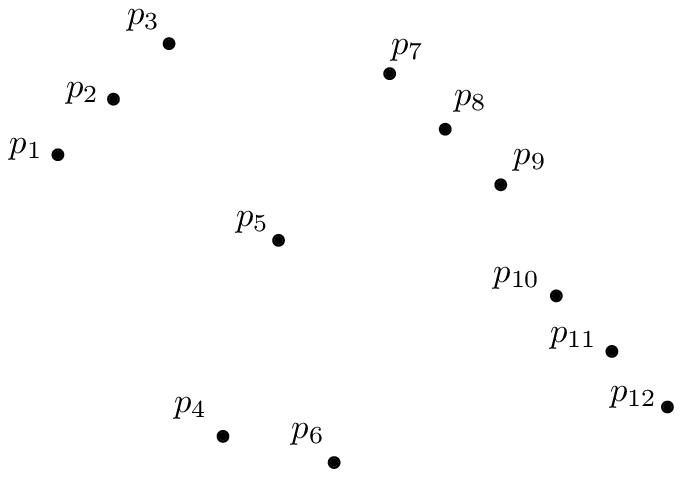,width=4cm,silent=}}}
        \hspace{15mm}
    \subfigure [Theorem \ref{dualbrect}]{ \label{bexample}
    			        {\epsfig{file=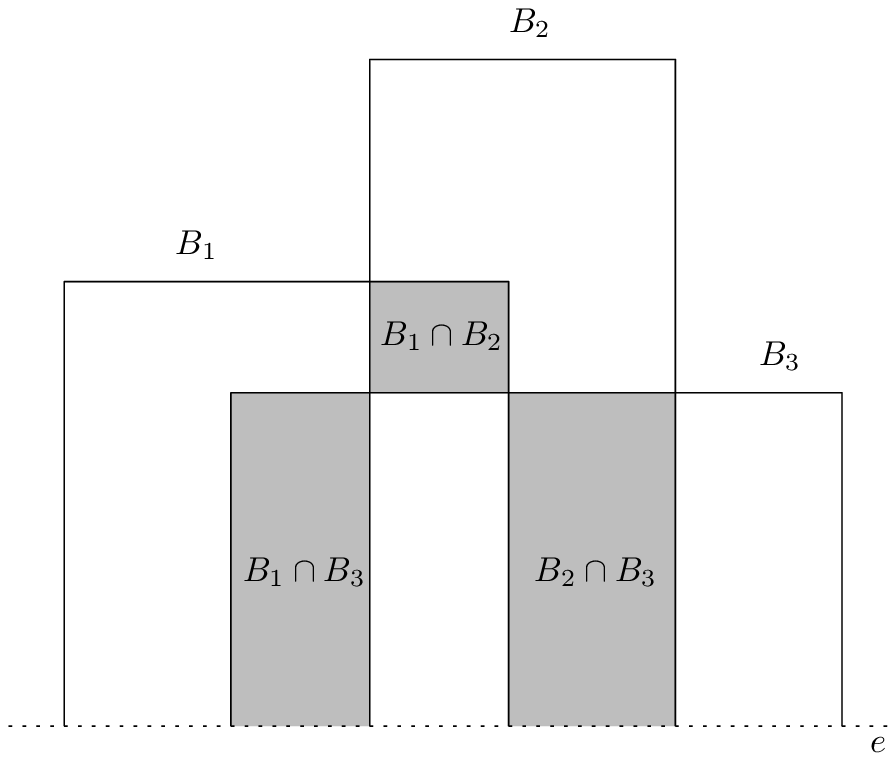,width=4cm,silent=}}}
        \hspace{15mm}
    \subfigure [Claim \ref{dualbrectprime}]{ \label{dualbrectprimefig}
    			        {\epsfig{file=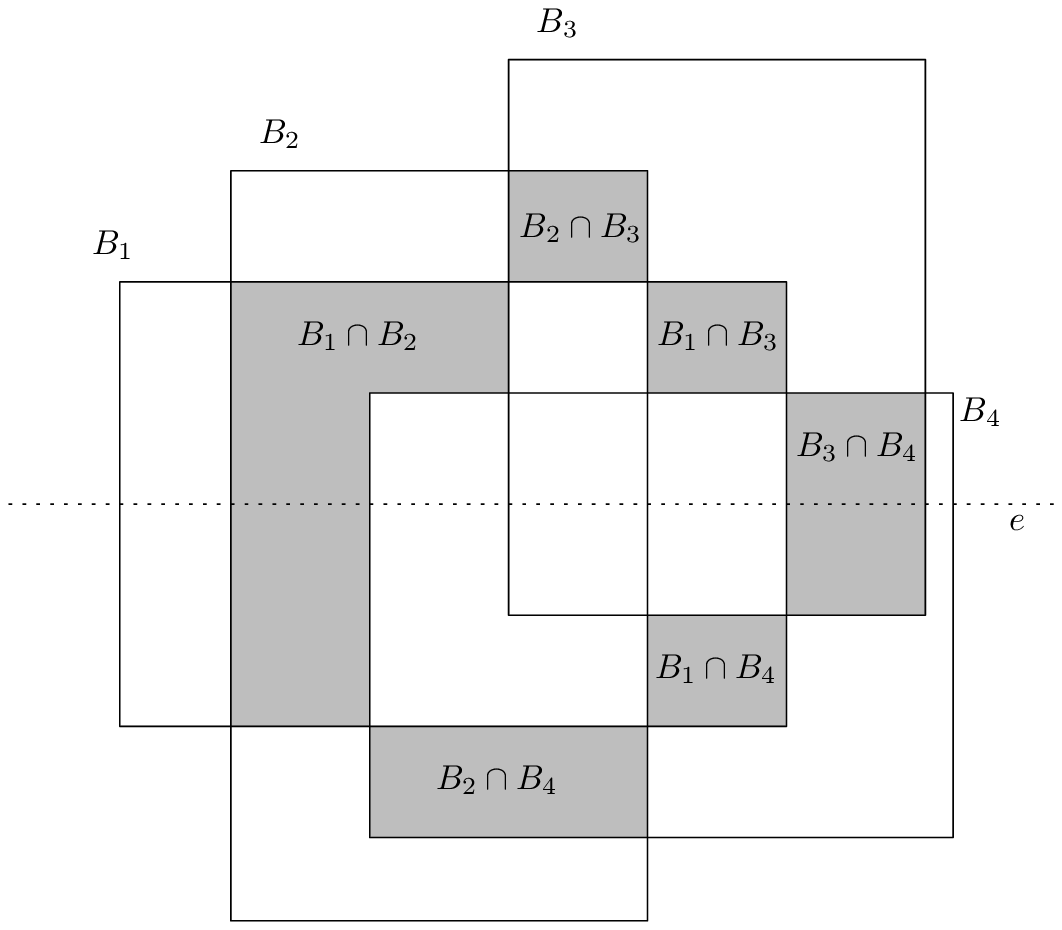,width=4cm,silent=}}}
    \caption {Lower bound constructions for bottomless rectangles}
    \label {figb}
\end{figure}
\end{proof}
The following theorem shows that the smallest $k$ for which $\chi_{k}({\cal B})=2$ is $4$ and so $\chi_{k}({\cal B})$ is determined for every $k$ as trivially $\chi_{k}({\cal B})\geq 2$ for any $k$. 
\begin{Theorem} \label{wcfbthm2}
$\chi_{4}({\cal B})=2$, i.e., any finite set of points can be colored with $2$ colors such that any bottomless rectangle containing at least $4$ of them is not monochromatic.
\end{Theorem} 
\begin{proof}      
We want to color the points red and blue such that any bottomless rectangle containing at least $4$ points contains two differently colored points. 
The proof is by induction. First we color the lowest point of $P$ arbitrarily then we consider the points in upwards order. We do not color every vertex as soon as it is considered but we always maintain the following induction hypothesis.

\begin{ih}\label{ih2}
In the $x$-coordinate order of the points $P'$ considered so far there are no two $x$-adjacent uncolored points and the colored points alternate in color (in the $x$-order).
\end{ih}

In a general step of the induction the next point in upwards order is considered. We keep it
uncolored unless it has an uncolored left or right neighbor in $P'$ (note that it cannot have both by IH \ref{ih2}). In that case we color the new point and its uncolored neighbor in a way that maintains the alternation. This way IH \ref{ih2} remains true. After we finish this process we arbitrarily color the points in $P$ that were left uncolored.
We need to prove that this coloring is a $4$-proper coloring. Consider a bottomless rectangle $B$ containing at least $4$ points. Let $p$ be the highest point contained by $B$. In the step when $p$ was considered $B\cap P$ was an interval of the points $P'$ considered so far. Any such interval of at least $4$ vertices contains both red and blue points as needed. Indeed, by IH \ref{ih2} any interval in the $x$-order with at least $4$ points has at least two colored points, thus by the alternation two differently colored points as well.
\end{proof}
\begin{claim} \label{wcfbthm3}
Colorings guaranteed by Claim \ref{folklore}, Claim \ref{wcfbthm} and Theorem \ref{wcfbthm2} can be found in $O(n\log n)$ time.
\end{claim} 
\begin{proof}
Computing the upwards order of the points takes $O(n\log n)$ time, the rest of the algorithm has $n$ steps in all cases, each computable in $O(\log n)$ time, in the algorithm of Theorem \ref{wcfbthm2} there is a final coloring step that takes at most linear time, so the whole algorithm runs in $O(n\log n)$ time in all cases.
\end{proof}

\subsection{Coloring bottomless rectangles}

In \cite{smorodnew} a very similar family is considered, namely, the family of axis-parallel rectangles intersecting a common base-line. The proof of their result with a slight modification gives $\ol{\chi_2}({\cal B})\leq 4$. The following theorems determine $\ol{\chi}_{k}({\cal B})$ for every $k$, also improving this bound to $3$ colors, which is optimal.

From now on we assume that there are no two bottomless rectangles with overlapping sides. It is easy to show that if this is not the case, then coloring the rectangles after perturbing them such that afterwards there are no overlappings, gives a needed coloring for the original family of rectangles as well.
We denote the projection of a point $p$ on $e$ by $p'$.

\begin{Theorem} \label{dualbrect}
$\ol{\chi_2}({\cal B})=3$, i.e., any finite family of bottomless rectangles can be colored with $3$ colors such that any point contained by at least $2$ of them is not monochromatic.
\end{Theorem} 

\begin{proof}
For the lower bound, the arrangement of $3$ rectangles $B_1,B_2,B_3$ on Figure \ref {bexample} shows that $3$ colors are sometimes needed. Indeed, for every pair of rectangles there is a point which is contained by only these two, thus all rectangles must have different colors in a $2$-proper coloring. 

For the upper bound, given a family of rectangles with a common base line $e$ we want to color the rectangles red, blue and green such that any point contained by at least $2$ rectangles is contained by two differently colored rectangles. The proof is by induction. We color the rectangles in downwards order according to their top edge's $y$-coordinate. During the process we may recolor already colored rectangles. We start with the empty family and reinsert the rectangles in this order.
We color the first, i.e., the highest rectangle blue. After each step we maintain the following induction hypotheses.
\begin{ih} \label {ih3}
Any point $p$ contained by at least $2$ rectangles is contained by two differently colored rectangles (it is a $2$-proper coloring).
\end{ih}
If a base-line point $q$ is contained by only $1$ rectangle then we say that the color of $q$ is the color of the rectangle containing it.
\begin{ih} \label {ih4}
If a point $q$ on the base-line $e$ is contained by exactly $1$ rectangle, then the color of $q$ is not red.
\end{ih}

In each step we insert the next rectangle $B$ in downwards order, so its top edge is below the top edge of all the rectangles already inserted. 
We color $B$ red (later $B$ might be recolored in order to maintain IH \ref{ih4}). We have to prove that IH \ref{ih3} remains valid for every point of the plane, i.e, this is again a $2$-proper coloring. 
Take an arbitrary point $p$, we prove that IH \ref{ih3} holds for this $p$.
\item[Case 1.] {\em $p$ is not in $B$.} Then IH \ref{ih3} already holds for $p$ by induction.
\item[Case 2.] {\em $p$ is contained by $B$ and at least $2$ rectangles besides $B$.} By induction not all these  rectangles have the same color, because IH \ref{ih3} was valid for $p$ before this step. Thus IH \ref{ih3} still holds for $p$. 
\item[Case 3.] {\em $p$ is contained by $B$ and by exactly one more rectangle.} For a point $p$ contained by $B$, the set of rectangles containing $p$ is the same as those containing its projection $p'$, because $B$ is the lowest rectangle. Thus IH \ref{ih3} holds for $p$ whenever it holds for $p'$. We know that $p'$ (like $p$) is contained by the red rectangle $B$ and by exactly one more rectangle, which is not red, because IH \ref{ih4} was valid for $p'$ before this step. Thus IH \ref{ih3} holds for $p'$ and thus for $p$ again.

If there is no base-line point contained by only $B$, then IH \ref{ih4} remains true. If there is such a point $q$ then we need to maintain the validity
of IH \ref{ih4}. 
It is easy to check that the following is true.
\begin{lemma}[divide and color] \label{recolorclaim}
If a base-line point $q$ is contained by only one rectangle with color $c$, then by switching the other two colors on the rectangles strictly left (resp. right) of $q$, IH \ref{ih3} remains valid. 
\end{lemma}
With only such `divide and color' steps we will change the coloring such that at the end there will be no point on the base-line contained by exactly $1$ green rectangle. Finally we will switch the colors green and red on all the rectangles to regain a coloring satisfying IH \ref{ih4}, while IH \ref{ih3} remains true during this process by Lemma \ref{recolorclaim}. For an illustration for the rest of the proof, see Figure \ref{divide}.
\begin{figure}
    \centering
            {\epsfig{file=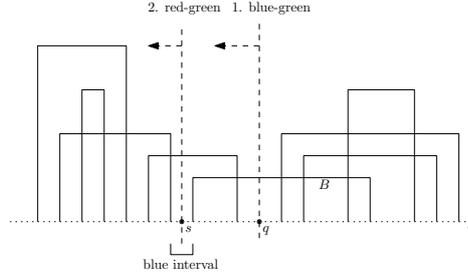,width=8cm,silent=}}
        \hspace{1mm}
    \caption{The color switches of the `divide and color' method in Theorem \ref{dualbrect}.}
      \label{divide}
\end{figure}

{\em Recoloring rectangles so that at the end there are no green base-line points.}
In the current coloring all green base-line points are left or right of $B$ because $B$ is red. Recall that $q$ is a base-line point contained only by $B$. We will deal with the left side first, changing the colors only of rectangles strictly left of $q$ and making a good coloring satisfying IH \ref{ih4} for any base-line point left of $q$. On the right side we proceed analogously, changing the colors only of rectangles strictly right of $q$.

{\em Recoloring rectangles strictly left of $q$.}
On the base-line left of $B$ there are some intervals of single colored points, all of them green or blue. 
\item[Case 1.] {\em If there is no green interval left of $q$}, i.e., there are no green base-line points left of $q$, we are done. 
\item[Case 2.] {\em If the closest such interval to $B$ is green}, we switch colors blue and green on the rectangles strictly left of $q$. This way the closest such interval to $B$ is now blue, we can proceed as in the next case.
\item[Case 3.] {\em If the closest such interval to $B$ is blue}, we switch colors red and green on the rectangles strictly left of any point $s$ of this blue interval. This way we got rid of all green base-line points left of $q$.

We recolor rectangles strictly right of $q$ in an analogous way.
At the end the coloring satisfies IH \ref{ih4} for any base-line point right of $q$. Thus we got a coloring satisfying both IH \ref{ih3} and IH \ref{ih4}.
\begin{figure}
    \centering
            {\epsfig{file=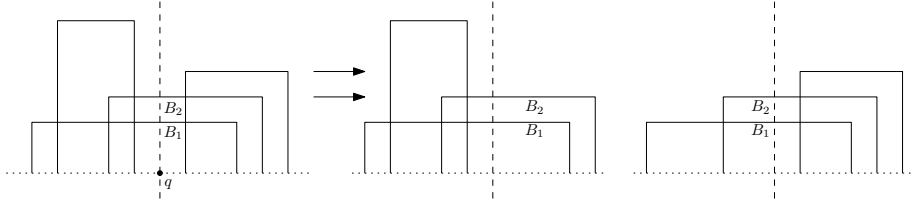,width=12cm,silent=}}
        \hspace{1mm}
    \caption{The division of the `divide and color' method in Theorem \ref{dualbrect2}.}
      \label{divide2}
\end{figure}
\end{proof}
\begin{Theorem} \label{dualbrect2}
$\ol{\chi_3}({\cal B})=2$, i.e., any finite family of bottomless rectangles can be colored with $2$ colors such that any point contained by at least $3$ of them is not monochromatic.
\end{Theorem} 
\begin{proof}

Given a family of rectangles with a common base line $e$ we want to color them red and blue such that any point contained by at least $3$ rectangles is contained by two differently colored rectangles. The proof is by induction on the number of rectangles. A single rectangle is colored red (it may be recolored later). We always maintain the following induction hypotheses.
\begin{ih} \label{ihx1}
Any point $p$ contained by at least $3$ rectangles is contained by a red and a blue rectangle as well (the coloring is a $3$-proper coloring).
\end{ih}
\begin{ih} \label{ihx2}
Any point $q$ on the base-line contained by exactly $2$ rectangles is contained by a red and a blue rectangle.
\end{ih}

We need a more advanced version of the `divide and color' method. In a general step, we distinguish some cases.

\item[Case 1.] {\em There is a point $q$ on the base-line not contained by any rectangle and there are some rectangles strictly left and strictly right of $q$.} Color the rectangles left of $q$ and the ones right of $q$ separately by induction, putting these together maintains the induction hypotheses.

\item[Case 2.]{\em There is a point $q$ on the base-line contained by exactly $1$ rectangle $B$ and there are some rectangles strictly left and strictly right of $q$.} Color first the rectangles left of $q$ together with $B$ then the rectangles right of $B$ together with $B$. This can be done by induction, because there were some rectangles on both sides. By a possible switching of the colors in the left and right parts, $B$ is red in both colorings. Putting together these two colorings we get a coloring of all the rectangles, for this coloring both induction hypotheses hold by induction.
\item[Case 3.]{\em There is a point $q$ on the base-line contained by exactly two rectangles, $B_1$ and $B_2$ and there is at least one rectangle both right and left of $q$.} See Figure \ref{divide2} for an illustration. In this case we can color by induction the rectangles strictly left of $q$ together with these two rectangles. As in this coloring IH \ref{ihx2} holds for $q$, $B_1$ and $B_2$ must have different colors. After a possible switch of the two colors we can assume that $B_1$ is red and $B_2$ is blue. The same way we can color by induction the rectangles strictly right of $q$ together with these two rectangles. This way the two rectangles are colored with the same colors in both colorings and so we can put together these two colorings, and we get a coloring of all the rectangles. For this coloring both induction hypotheses hold by induction.

\item[Case 4.]{\em For any base-line point $q$ contained by exactly $1$ or $2$ rectangles, there is no rectangle strictly left or right of $q$.} The left and right sides of the rectangles divide the base-line into $2$ half-lines and $2n-1$ intervals. It is easy to see that in this case the only base-line points contained by exactly $1$ rectangle are the points of the leftmost $L_1$ and rightmost $R_1$ interval and the 2-contained points are the points of the second leftmost $L_2$ and second rightmost $R_2$ interval.
Consider rectangle $B$, the one with the lowest top edge. Again we have to distinguish some cases.
\item[Case 4a.] {\em $B$ does not contain $1$- or $2$-covered base-line points.} In this case we color the rest of the rectangles by induction and then color $B$ with an arbitrary color. This way IH \ref{ihx2} remains valid trivially. Take an arbitrary $p$ contained by at least $3$ rectangles, we prove that IH \ref {ihx1} holds for $p$. 

{\em If $p$ is not contained by $B$}, then IH \ref {ihx1} holds for $p$ by induction.

{\em If $p$ is contained by $B$ and at least $3$ other rectangles}, then IH \ref {ihx1} holds for $p$ by induction.

{\em If $p$ is contained by $B$ and exactly $2$ other rectangles}, then its projection $p'$ is contained by $B$ and the same $2$ rectangles, because $B$ was the lowest rectangle. Before this step IH \ref{ihx2} was valid for $p'$, thus these two rectangles have different colors.

In the rest of the cases $B$ contains some intervals from $L_1,L_2,R_1,R_2$. 
\item[Case 4b.] {\em $B$ contains some of $L_1$ and $L_2$ and does not contain $R_1$ and $R_2$}. Color the rest of the rectangles by induction and then color $B$ with a color different from the other rectangle containing $L_2$. This way we obtain a coloring that satisfies IH \ref{ihx2} by induction. The coloring satisfies IH \ref{ihx1} as well, the proof is similar to the one in the previous case.

\item[Case 4c.] {\em $B$ contains some of $R_1$ and $R_2$ and does not contain $L_1$ and $L_2$}. The coloring is symmetrical to the one in the previous case. The proof is also analogous.

In the rest we have to distinguish cases according to the rectangle $B_2$ with the second lowest top edge.
\item[Case 4d.] {\em $B$ contains $L_2$ and $R_2$ as well and $B_2$ does not contain any of $L_1,L_2,R_1,R_2$.} In this case color all the rectangles except $B_2$ by induction and then color $B_2$ with the same color as $B$. For this coloring IH \ref{ihx2} holds as it is enough to check the points of $L_2$ and $R_2$ and here IH \ref{ihx2} holds by induction. We have to prove that IH \ref{ihx1} holds as well. For that take an arbitrary $p$ contained by at least $3$ rectangles, we prove that IH \ref{ihx1} holds for $p$, i.e., not all of these rectangles have the same color. 

{\em If $p$ is not contained by $B_2$}, then IH \ref {ihx1} holds for $p$ by induction.

{\em If $p$ is contained by $B_2$ and at least $3$ other rectangles}, then IH \ref {ihx1} holds for $p$ by induction. 

{\em If $p$ is contained by $B_2$ and exactly $2$ other rectangles and its projection $p'$ is contained by the same rectangles as $p$}, then by IH \ref {ihx2} $p'$ is contained by two differently colored rectangles besides $B_2$. This holds for $p$ as well. 

{\em If $p$ is contained by $B_2$ and exactly $2$ other rectangles and its projection $p'$ is not contained by the same rectangles as $p$}, then the only possibility is that $p'$ is contained by $B$ too (as only $B$ is lower than $B_2$). By induction IH \ref {ihx1} holds for the rectangles excluding $B_2$, thus $p'$ was contained by red and blue rectangles as well without considering $B_2$. As $B_2$ has the same color as $B$, the same holds for $p$.

\item[Case 4e.] {\em $B$ contains $L_2$ and $R_2$ as well and $B_2$ contains some of $L_1,L_2,R_1,R_2$.} By symmetry we can assume that $B_2$ contains $L_2$ (and maybe $R_2$ too). In this final case delete both $B$ and $B_2$ and color the rest of the rectangles by induction. Now put back these two rectangles. If $R_2$ is contained by some rectangle besides $B$ and $B_2$ then color $B$ differently from the color of this rectangle. Otherwise color $B$ arbitrarily. Finally, color $B_2$ differently from $B$. 
In this coloring IH \ref {ihx2} holds as it is enough to check the points of $L_2$ and $R_2$ and here IH \ref {ihx2} holds trivially. We prove that IH \ref{ihx1} holds as well. For that take again an arbitrary $p$ contained by at least $3$ rectangles, we prove that IH \ref{ihx1} holds for $p$, i.e., not all of these rectangles have the same color.

{\em If $p$ is contained by  both $B$ and $B_2$}, then IH \ref{ihx1} holds for $p$ as these rectangles are colored differently. 

{\em If $p$ is contained by at least $3$ rectangles besides $B$ and $B_2$}, then again by induction IH \ref{ihx1} holds for $p$.

{\em If $p$ is contained by one of $B$ and $B_2$ and only two other rectangles}, then its projection $p'$ was contained by exactly two rectangles $B_3$ and $B_4$ in the coloring without $B$ and $B_2$. $B_3$ and $B_4$ must have different colors, because by induction IH \ref{ihx2} holds for the rectangles excluding $B$ and $B_2$. As $B$ and $B_2$ are the lowest rectangles, the point $p$ is contained by the differently colored rectangles $B_3$ and $B_4$ as well.
\end{proof}
\begin{claim} \label{dualbrect3}
Colorings guaranteed by Theorem \ref{dualbrect} and Theorem \ref{dualbrect2} can be found in $O(n^2)$ time.
\end{claim} 
\begin{proof}
Finding the upwards order of the rectangles takes $O(n\log n)$ time. In each step we maintain an array of the intervals of the base line. If an interval is contained only by one rectangle, we keep its color as well.

In the algorithm of Theorem \ref{dualbrect} in each step we search for some colored interval constant times and recolor some rectangles with a given property (left of a given interval, etc.) constant times. This takes $c\cdot k$ time if we have $k$ rectangles at that step. We have $n$ such steps and $k\leq n$ always, so the running time is $O(n^2)$.

For the algorithm of Theorem \ref{dualbrect2}, we prove by induction on the number of bottomless rectangles, that $c_0\cdot n^2$ is an upper bound on the number of steps needed to color any family of $n$ bottomless rectangles, for some $c_0$ large enough. Except Case $4$ we always do the `divide and color' step by cutting the family into two nontrivial parts and color separately. Finding whether there is such a cut, doing the cut (and maintaining the upwards order in the two parts) and the possible recolorings after the recursional colorings take $c_1\cdot n$ time for $n$ rectangles. By induction, this and the two recursional algorithms together take at most $c_0\cdot a^2+c_0\cdot b^2+c_1\cdot n$ time where $a+b\leq n+2$. 
In Case $4$ we can decide which kind of step is needed and color $B$ and $B_2$ in $c_2\cdot n$ steps, and we can do the recursion in $c_0\cdot (n-1)^2$ steps. Thus we need at most $c_0\cdot(n-1)^2+c_2n$ time in this case. 
It is easy to see that in both of these cases the time can be bounded from above by $c_0\cdot n^2$ if $c_0$ is large enough (depending on $c_1$ and $c_2$).
\end{proof}

Consider now the case of axis-parallel rectangles intersecting a common base-line $e$ (this family is denoted by $\cal B'$). We start with the case of region coloring, that is, estimating $\ol{\chi_k}({\cal B'})$. For this the best lower bound is due to \cite{pachtothtardos}, where they give a construction showing that $\ol{\chi_2}({\cal B'})\geq 3$ for every $k$ (actually their construction is for arbitrary axis-parallel rectangles but it can be easily modified to use only axis-parallel rectangles intersecting a common base-line). The best upper bound is due to \cite{smorodnew}, proving $\ol{\chi_2}({\cal B'})\leq 8$, and for the case of $k>2$ we can separately color the upper and lower parts (divided by the base-line) of the rectangles with $2$ colors by Theorem \ref{dualbrect2} and then for a rectangle colored by $a$ in the upper part and $b$ in the lower part, we give the ordered pair $(a,b)$ as a color. It is easy to see that this is a good $3$-proper coloring of the rectangles, thus proving:
\begin{claim} \label{bprimed}
$\ol{\chi_3}({\cal B'})\leq 4$. 
\end{claim}
Further there is a simple construction of four rectangles where for each two rectangles there is a point covered by only these two rectangles (see Figure \ref{dualbrectprimefig}), this implies the following.
\begin{claim}\label{dualbrectprime}
$\ol{\chi_2}({\cal B'})\geq4$. 
\end{claim}
The case of coloring points seems less natural for axis-parallel rectangles intersecting a common base-line, still it can be considered. Coloring the points in the lower and upper parts with different colors ensures that any rectangle covering one from both sides is not monochromatic. The two sides can be colored by Claim \ref{folklore} with $3$-$3$ colors, thus proving:
\begin{claim}
$\chi_2({\cal B'})\leq 6$. \label{bprime2}
\end{claim}
Indeed, a rectangle either contains points from both sides or contains at least $2$ points on one 
side. Further, this also implies that:
\begin{claim} \label{bprime3}
$\chi_3({\cal B'})\leq3$. 
\end{claim}
Indeed, color both sides with the same $3$ colors according to Claim \ref{folklore}, 
then any rectangle containing at least $3$ points contains $2$ points on one side, thus containing two differently colored ones as well. Finally, we have that: 
\begin{claim}
$\chi_7({\cal B'})=2$. \label{bprime4}
\end{claim}
Indeed, if we color both sides with the same two colors according to Theorem \ref{wcfbthm2}, then any rectangle containing at least $7$ points contains $4$ point on one side, thus containing a red and blue one as well. The lower bounds for bottomless rectangles trivially hold for the case of $\cal B'$ as well, thus Claim \ref{folklore} implies that:
\begin{claim}
$\chi_2({\cal B'})\ge 3$. \label{bprimel}
\end{claim}
\section{Half-planes} \label {sechp}

The family of all half-planes is denoted by $\cal H$. We prove exact bounds for ${\chi_k}({\cal H})$ and almost exact bounds for $\ol{\chi_k}({\cal H})$.

From now on we assume that there are no $3$ points on one line. It is easy to show that if this is not the case, then coloring the point set after a small perturbation gives a needed coloring for the original point set as well.
This way the vertices of the convex hull of a point set $P$ are exactly the points of $P$ being on the boundary of this convex hull.
\subsection{Coloring points}

The following lemma follows easily from the definition of the convex hull.
\begin{lemma} \label{hplemma}
Any half-plane $h$ containing at least one point of $P$ contains some vertex of the convex hull of $P$ too. Moreover, the vertices of the convex hull of $P$ contained by $h$ are consecutive on the hull.
\end{lemma}

\begin{Theorem} \label{hptheorem}
\item[(i)]
$\chi_2({\cal H})=4$, i.e., any finite set of points can be colored with $4$ colors such that any half-plane containing at least $2$ of them is not monochromatic, and $4$ colors might be needed.
\item[(ii)]
$\chi_2({\cal H},P)\leq 3$, except when $P$ has $4$ points, with one of them inside the triangle determined by the other $3$ points (see Figure \ref{fighpl}), in which case $\chi_2({\cal H},P)=4$.
\end{Theorem}

\begin{figure}
    \centering
    \subfigure[The exceptional case of Theorem \ref{hptheorem}(ii)]{\label{fighpl}
        {\epsfig{file=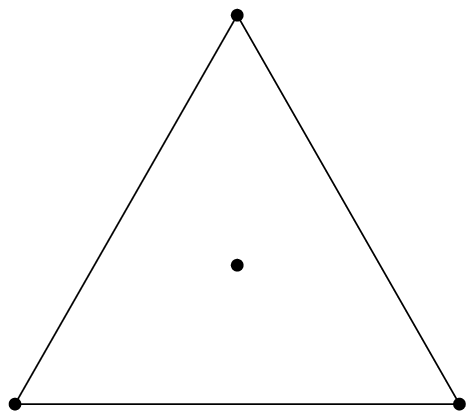,width=5cm,silent=}}}
        \hspace{10mm}
    \subfigure[The proof of Theorem \ref{hptheorem}(ii)]{\label{fighpproof}
        {\epsfig{file=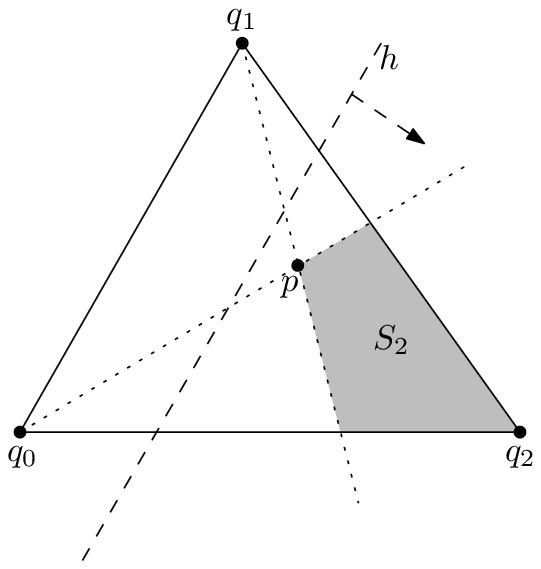,width=5cm,silent=}}}
   \caption{Theorem \ref{hptheorem}(ii)}
   \label{fighp}
\end{figure}

\begin{proof}
$(i)$
This follows from $(ii)$, yet we give a short proof for the upper bound.
Color the vertices of the convex hull of $P$ with $3$ colors such that there are no $2$ vertices next to each other on the hull with the same color. Color all the remaining points with the $4$th color. This coloring is good as by Lemma \ref{hplemma} any half-plane containing at least two points contains two neighboring vertices on the hull or one vertex on the hull and one point inside.

$(ii)$
Clearly, in the case when $P$ has $4$ points, with one of them inside the triangle determined by the other $3$ points (we denote this special case by $P^*$), we need four colors to have a $2$-proper coloring as any two points can be contained by some half-plane not containing the rest of the points.

As $\cal H$ is monotone, by Observation \ref{monotone} it is enough to consider half-planes containing exactly $2$ points of $P$. I.e., we need a coloring such that for any $h$ containing exactly two points, $h$ is not monochromatic. We color the vertices of the convex hull with $3$ colors as in (i) such that there are no $2$ vertices next to each other on the hull with the same color. At this time every half-plane $h$ containing two vertices of the convex hull is already not monochromatic, because by Lemma \ref{hplemma} such a $h$ contains two neighboring vertices also, which do not have the same color. Now we color the points inside the hull in a more clever way than in (i). Take an arbitrary point $p$ inside the hull. The only case when a half-plane $h$ contains exactly one point of $P$ besides $p$, is when $h$ contains only $p$ and one vertex $q$ of the convex hull. In such a case we say that $q$ {\em can be cut out} with $p$ (by $h$). 
If this can happen only with two vertices of the hull, then coloring $p$ different from these two points, all half-planes containing $p$ will be not monochromatic. Doing the same for every inside point we get a $2$-proper coloring of $P$.

Denote the vertices of the convex hull by $q_0,...,q_{k-1}$ in clockwise order (indexes are mod $k$). Given a $p$, we prove that there are no $3$ vertices on the hull that can be cut out with $p$ except for the special case $P^*$. First, notice that if  $q_i$ can be cut out with $p$, then $p$ is inside the triangle $T_i=q_{i-1}qq_{i+1}\Delta$. It is easy to see that if the hull has more then $3$ vertices, then there are no $3$ such triangles having a common inner point. For the rest of the proof see Figure \ref{fighpproof}. If the hull has $3$ vertices and $p$ can be cut out with all these $3$ vertices, then regard the lines going through some $q_i$ and $p$ partitioning the triangle into $6$ triangles. For each vertex $q_i$, denote the union of the two triangles having it as a vertex by $S_i$. Thus, we have three quadrilaterals, all of which must be empty. Indeed, e.g. $q_2$ can be cut out with $p$ by a half-plane $h$. By definition $h$ is not containing any other point of $P$, yet $h$ always contains the quadrilateral $S_2$, and so $S_2$ must be empty. The same argument for the other two quadrilaterals shows that all of them must be empty and so $p$ is the only point in the triangle, which is the excluded case $P^*$.
\end{proof}

\begin{Theorem} \label{hptheorem2}
$\chi_3({\cal H})=2$, i.e., any finite set of points can be colored with $2$ colors such that any half-plane containing at least $3$ of them is not monochromatic.
\end{Theorem}

\begin{proof}
As $\cal H$ is monotone, it is enough to consider half-planes containing exactly $3$ points of $P$. We color the points with colors red and blue.
The points inside the convex hull of $P$ are colored blue.
The vertices of the convex hull of $P$ are denoted by $q_0,\ldots,q_{k-1}$ in clockwise order. For each $q_i$ we assign $T_i=q_{i-1}q_iq_{i+1}\Delta$, where indexes are modulo $k$. If $T_i$ has some point of $P$ inside it, then color $q_i$ red. 

If there are no nonempty $T_i$'s then color the vertices of the convex hull with alternating colors, if its size is odd, then with the exception of two neighboring red points. 
If there is at least one nonempty $T_i$ then these red points cut the boundary of the convex hull into chains. For each chain color its vertices with alternating colors, a chain with size one is colored blue. This coloring has the following property.

\begin{obs}
There are no $2$ consecutive blue vertices on the convex hull.
\end{obs}
We prove that the  coloring defined above is a $3$-proper coloring. Take again an arbitrary half-plane $h$ containing exactly $3$ points. We prove that it contains a red and a blue point too. By Lemma \ref{hplemma} $h$ contains some consecutive vertices of the convex hull of $P$. 
\item [Case 1.] {\em If $h$ contains at least two consecutive vertices on the hull}, then it contains at least one red point. 
\item [Case 1a.] {\em If $h$ contains at least one point inside the hull}, then it contains at least one blue point. 
\item [Case 1b.] {\em If $h$ contains three vertices of the convex hull but no points inside}, then it is easy to see that the triangle corresponding to the middle point $q$ in the ordering must be empty. So $q$ belongs to some alternatingly colored chain. If any of its neighbors $q'$ corresponds to the same chain, then $h$ contains the points $q$ and $q'$ that have different colors by definition. If $q$ is a chain of size $1$ then it is blue and its neighbors are red, again good.
\item [Case 2.] {\em If $h$ contains one vertex $q_i$ of the convex hull and two points of $P$ inside the convex hull}, then the latter points are blue and they must be in $T_i$, that is $q_i$ is red. Thus $h$ contains red and blue points as well.
\end{proof}

\begin{obs} \label{hpobs}
The algorithm in the proof of Theorem \ref{hptheorem2} gives a coloring which additionally guarantees that there are no half-planes containing exactly two points, both of them blue.
\end{obs}

\begin{claim} \label{hptheorem3}
Colorings guaranteed by Theorem \ref{hptheorem} and Theorem \ref{hptheorem2} can be found in $O(n \log n)$ time.
\end{claim}

\begin{proof}
The algorithm in Theorem \ref{hptheorem} $(i)$ clearly works in $O(n\log n)$, the same as building the convex hull.
For the other two algorithm we need the dynamic convex hull algorithm presented in \cite{brodal}.

For the algorithms in Theorem \ref{hptheorem2} we first compute a convex hull in $O(n\log n)$ amortized time and then we take its points one by one and do the following. Delete temporarily the convex hull vertex $p$, compute the new convex hull temporarily, if it has some new vertices on it, then the triangle corresponding to $p$ is not empty. As any inner point has been added and deleted from the set of vertices of the hull at most two times and the convex hull algorithm makes a step in $O(\log n)$ amortized time, we could decide in $O(n\log n)$ time which vertices of the hull have empty triangles. After that the coloring of the vertices of the hull and the inside points takes $O(n)$ time, $O(n\log n)$ altogether.

For the algorithm in Theorem \ref{hptheorem} $(ii)$ we do the same just when we temporarily delete $p$ we assign to any additional convex hull vertex the point $p$, as this vertex can be cut out by a half-plane together with $p$. After these we simply color the vertices of the convex hull as needed and all the inner points with a color different from the color of the at most two convex hull vertices assigned to it. Altogether this is again $O(n\log n)$ time.
\end{proof}

\subsection{Coloring half-planes}

First we introduce some tools necessary for the proofs of this section.
We can assume that there are no half-planes with vertical boundary line.
We dualize the half-planes and points of the plane $S$ with the points (with an additional orientation) and lines of plane $S'$, then we color the set of directed points corresponding to the half-planes which will give a good coloring of the original family of half-planes. The dualization is as follows. For a half-plane $h$ with a boundary line given by the equality $y=ax+b$ the corresponding dual point $h^*$ has coordinates $(a,b)$. If this line is a lower boundary, then $h^*$ has {\em orientation} {\em north}, otherwise it has orientation {\em south}. For an arbitrary point $p$ with coordinates $(c,d)$ the corresponding line $p^*$ is given by $y=-cx+d$. Now it is easy to see that $h$ contains $p$ on the primal plane if and only if the vertical ray starting in $h^*$ and going into its orientation meets line $p^*$ (we say that $h^*$ and $p^*$ {\em see each other} or $h^*$ is {\em looking at} $p^*$). Indeed, for a half-plane with lower boundary both hold if and only if $d>ac+b$, for a plane with an upper boundary both hold if and only if $d<ac+b$. From these it follows that:
\begin{obs} \label{dualobs}
The $k$-proper coloring of half-planes is equivalent to a coloring of the dual set of oriented points such that any line with at least $k$ points looking at it, there are at least two among these at least $k$ points with different colors.
\end{obs}

All the proofs give colorings for directed points and from now on we assume that there are no $3$ directed points on one line. It is easy to show that if this is not the case, then coloring the set of directed points after a small perturbation gives a needed coloring for the original set of directed points as well.

\begin{Theorem} \label{halfplanes}
$\ol{\chi_2}({\cal H})=3$, i.e., any finite family of half-planes can be colored with $3$ colors such that any point contained by at least $2$ of them is not monochromatic.
\end{Theorem}

\begin{proof}[Proof of Theorem \ref{halfplanes}.]
\begin{figure}
    \centering
    \subfigure[Construction for Theorem \ref{halfplanes}]{\label{fighpproof2}
    {\epsfig{file=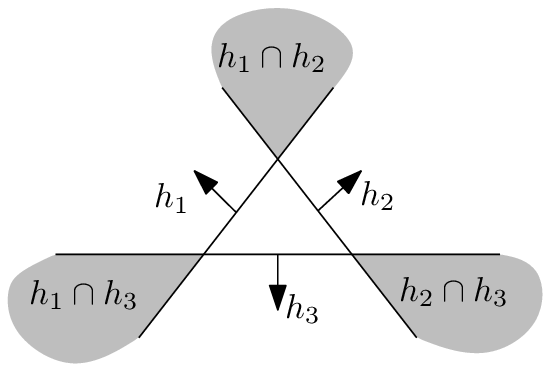,width=6cm,silent=}}}
        \hspace{10mm}
		\subfigure[Proof of Theorem \ref{halfplanes}]{\label{fighpproof3}
    {\epsfig{file=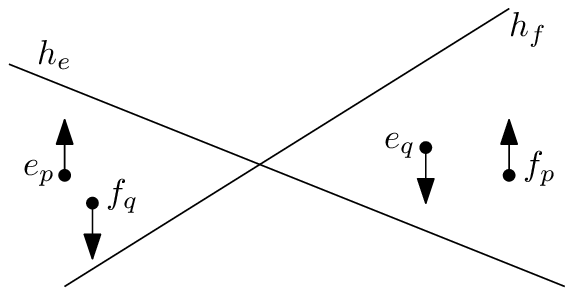,width=6cm,silent=}}}
   \caption{Theorem \ref{halfplanes}(i)}
   \label{hpdualfigs}
\end{figure}
For a construction proving that $3$ colors might be needed, see Figure \ref{fighpproof2}. Here every pair of half-planes has a point contained by exactly these two half-planes, thus all $3$ half-planes must have different colors in a $2$-proper coloring.

For the upper bound, by Observation \ref{dualobs} we can color directed points with respect to lines. Given a set of directed points $P$ we have to color $P$ with $3$ colors such that for any line $h$ seeing at least $2$ points, not all of these points have the same color.
Take the lower boundary of the convex hull of the set of north-directed points and denote the vertices of it by $p_1,p_2,\ldots,p_k$ ordered by their $x$-coordinate. Take the upper boundary of the convex hull of the set of south-directed points and denote the vertices of it by $q_1,q_2,\ldots,q_l$ ordered by their $x$-coordinate. The rest of the points we call {\em inner} points. Similarly to Lemma \ref{hplemma} any line seeing at least one north-directed point sees one $p_i$ as well and any line seeing at least one south-directed point sees one $q_j$ as well and the $p_i$'s and $q_i$'s seen by a line are consecutive. 

First we give a coloring of the $p_i$'s and $q_j$'s with $3$ colors such that no two consecutive points have the same color and if for some $p_i$ and $q_j$ there is a line which sees exactly these two points, then these points have different colors. Any line $h$ which does not see inner points and sees at least two points, sees either exactly one $p_i$ and $q_j$ or at least two consecutive points of the same type. Thus the coloring will be good for all such lines.

We define a graph $G$ on the points $p_i$ and $q_j$. The consecutive points are connected forming a path of $p$'s and a path of $q$'s. Moreover, $p_i$ and $q_j$ are connected if there is a line which sees exactly these two points. Clearly, we need a proper $3$-coloring of this graph. For algorithmic reasons we take a graph with more edges and prove that it can be $3$-colored as well. In this graph $p_i$ and $q_j$ are connected if there is a line which sees no other points of the $p$-path and $q$-path. We claim that drawing the $p$-path and the $q$-path on two parallel straight lines, the $q$-path being on the higher line and in reverse order, and drawing all the edges with straight lines, we have a graph without intersecting edges. In other words the graph is a caterpillar-tree between two paths. Such a graph is outer-planar and thus three-colorable.

It is left to prove that there are no intersecting edges in the above defined drawing of $G$. Without loss of generality such two edges $e$ and $f$ would correspond to points $e_p$, $e_q$, $f_p$ and $f_q$ with $x$-coordinates $e_p^x<f_p^x$ and $e_q^x<f_q^x$ (the points with index $p$ are from the $p$-path and the points with index $q$ are from the $q$-path). The line seeing only $e_p$ and $e_q$ is denoted by $h_e$, the line seeing only $f_p$ and $f_q$ is denoted by $h_f$. These two lines divide the plane into four parts, which can be defined as the north, south, west and east part. Clearly from $e_p$ and $f_p$ one must be in the west part and one in the east part. By $e_p^x<f_p^x$, $e_p$ is in the west and $f_p$ is in the east part. This means that $h_e$ must be the line above the east and south parts and so $e_q$ must be in the east part and $f_q$ in the west, a contradiction together with $e_q^x<f_q^x$ (see Figure \ref{fighpproof3}).

Now we finish the coloring such that the condition will hold also for lines seeing inner points.
As in Theorem \ref{hptheorem} $(ii)$ for any other north-directed point $p$ there are two points $p_i$ and $p_{i+1}$ (the unique ones for which $p_i$ has smaller and $p_{i+1}$ has bigger $x$-coordinate than $p$) such that whenever a line $h$ sees $p$ then it sees $p_i$ or $p_{i+1}$ as well. Coloring $p$ differently from these two points guarantees that any $h$ seeing $p$ sees two differently colored points. We do the same for the south-directed points. This way whenever a line $h$ sees some point which is not a $p_i$ or $q_j$ then it sees points with both colors. Earlier we proved this for lines that see at least two points but do not see inner points, so these two claims together prove that the coloring is a $2$-proper coloring.
 
\end{proof}
\begin{Theorem} \label{halfplanes2}
$\ol{\chi_4}({\cal H})=2$, i.e., any finite family of half-planes can be colored with $2$ colors such that any point contained by at least $4$ of them is not monochromatic.
\end{Theorem}

\begin{proof}
By Observation \ref{dualobs} we can color directed points with respect to lines.
Given a set of directed points we have to color them with $2$ colors such that for any $h$ line seeing at least $4$ points, not all of these points have the same color.
We color the north-directed points with the same algorithm as in Theorem \ref{hptheorem2}. We color the south-directed points with the same algorithm as in Theorem \ref{hptheorem2} just with inverted colors. Take now an arbitrary $h$ line that sees at least $4$ points, we have to prove that it sees red and blue points too.
\item[Case 1.] {\em If a line $h$ sees at least $3$ north-directed points}, then by the algorithm $h$ sees red and blue points as well.
\item[Case 2.] {\em If a line $h$ sees at least $3$ south-directed points}, then by the algorithm $h$ sees red and blue points as well.
\item[Case 3.] {\em If a line $h$ sees exactly $2$ points of each kind}, then there are two options. If $h$ sees red and blue points as well of one kind, then we are done. Otherwise, by Observation \ref{hpobs}, $h$ must see $2$ red north-directed points and $2$ blue south-directed points, again seeing points with both colors.
\end{proof}

\begin{claim} \label{halfplanes3}
Colorings guaranteed by the theorems can be found in $O(n \log n)$ time for Theorem \ref{halfplanes2} and in $O(n^2)$ time for Theorem \ref{halfplanes}.
\end{claim}
\begin{proof}
The algorithm in Theorem \ref{halfplanes2} clearly runs in time $O(n\log n)$ using Theorem \ref{hptheorem3}. The algorithm in Theorem \ref{halfplanes} can be made similarly to work in this time, only the building of the caterpillar tree might need $O(n^2)$ steps. Indeed, we just need to prove that deciding whether there is an edge between some $p_i$ and $q_j$, can be done in constant time. For that we need to check whether the linear equations for a line assuring that it goes above $q_{j-1}$, $q_{j+1}$ and below $q_j$, below $p_{i-1}$, $p_{i+1}$ and above $p_i$ have a solution. These can be checked in constant time.
\end{proof}

\section{Discussion} \label{secdiscussion}

Although for bottomless rectangles we have the best bounds, for axis-parallel rectangles intersecting a common base-line there are gaps between the lower and upper bounds (see Table \ref{table:bottomless}).
\begin{problem}
Give better bounds for $\ol{\chi_k}({\cal B'},n)$ and $\chi_k({\cal B'},n)$.
\end{problem}

As we mentioned for half-planes there is only one case left open.
\begin{problem}
Determine the value of $\ol{\chi_3}({\cal H})$, i.e., the fewest number of colors needed to color any finite family of half-planes such that if a point of the plane is contained by at least $3$ of them then not all of the containing half-planes have the same color.
\end{problem}

Note that $\ol{\chi_3}({\cal H})$ is either $2$ or $3$.

Let us also summarize the results known about axis-parallel rectangles and discs (the two most widely examined region types) and pose open problems regarding the unsolved cases.
The general case of axis-parallel rectangles (denoted by $\cal R$) is still far from being solved, the best bounds are $\chi_2({\cal R},n)=\Omega(\frac{\log n}{(\log \log n )^2})$ by Chen et al. \cite{chenpach} from below and recently by Ajwani et al. \cite{elbassioni} $\chi_2({\cal R},n)=\tilde{O}(n^{.382+\epsilon})$ from above, improving the previous bound $\chi_2({\cal R},n)=O(\sqrt{\frac{n \log \log n}{\log n}})$ by Pach et al. \cite{pachtoth}. So probably one of the most interesting problems is still to give better bounds for ${\chi_2}({\cal R},n)$, i.e., the fewest number of colors needed to color any set of $n$ points, such that if an axis-parallel rectangle contains at least two of them then not all of those contained by it have the same color.

For the dual case of coloring axis-parallel rectangles the proof of the upper bound ${\ol{cf}}({\cal R},n)=O(\log^2 n)$ (\cite{har}) can be modified easily to give the upper bound ${\ol{\chi_2}}({\cal R},n)=O(\log n)$. There is a matching lower bound ${\ol{\chi_2}}({\cal R},n)=\Omega(\log n)$ by Pach et al. \cite{tardosmanu} (they actually prove ${\ol{\chi_k}}({\cal R},n)=\Omega(\log n)$ for any fixed $k$). This implies the same lower bound for ${\ol{cf}}({\cal R},n)$, thus for conflict-free colorings there is still a slight gap here between the lower and upper bounds.

The case of discs in the plane (denoted by $\cal D$) is only partially solved. It is easy to see that a proper coloring of the Delaunay-triangulation of a point set $P$ gives a proper coloring of $H(P,\cal D)$. Indeed, every disc containing at least two points of $P$, contains two points connected by an edge in the Delaunay-graph. As this graph is planar, by the four-color theorem we can always color it with $4$ colors, thus we conclude that $\chi_2({\cal D})=4$. Further, Pach et al. \cite{pachtothtardos} showed that $\chi_k({\cal D})>2$ for any $k$. 
\begin{problem}
Is it true for some $k$ that $\chi_k({\cal D})=3$? I.e., is there a $k$ for which every finite set of points $P$ can be colored by $3$ colors such that if a disc contains at least $k$ points of $P$, then not all of them have the same color. If yes, find the smallest such $k$.
\end{problem}

Answering the question if $\ol{\chi_k}({\cal D})$ exists for some $k$, Smorodinsky \cite{smorodnew} showed that $\ol{\chi_2}({\cal D})=4$.

\begin{problem}
Give better bounds for $\ol{\chi_k}({\cal D})$, when $k>2$.
\end{problem}

\noindent {\bf Acknowledgment.} 
This paper is an extended version of \cite{kbcccg}. The results of the paper appear also in the PhD Dissertation of the author \cite{kbphd}. In these papers a different notation is used, $k$-proper colorings are called as $wcf_k$-colorings (weak conflict-free colorings).

The author is grateful to G\'abor Tardos for introducing this topic and for his many useful comments on this paper and also to Joseph O'Rourke for his advices on how to improve the presentation.



\end{document}